\documentclass[11pt,a4paper,reqno]{article}
\usepackage{amsmath}
\usepackage{amsthm}
\usepackage{enumerate}
\usepackage{amssymb}

\usepackage{verbatim}
\usepackage{url}
\usepackage[latin1]{inputenc}

\usepackage{caption}
\usepackage{graphicx,color}
\def\N{{\mathbb N}}
\def\Z{{\mathbb Z}}

\def\R{{\mathbb R}}

\def\Hb{{\mathbb H}}

\def\Bc{{\mathcal B}}

\def\Dc{{\mathcal D}}
\def\Ec{{\mathcal E}}
\def\Gc{{\mathcal G}}
\def\Hc{{\mathcal H}}

\def\Vc{{\mathcal V}}
\def\Wc{{\mathcal W}}

\def\eps{\varepsilon}
\def\VH{\Vc_\Hc}
\def\iVH{\Vc_\Hc^\circ}
\def\bVH{\partial\Vc_\Hc}

\def\muZ{\mu}
\def\Cyl{{\mathcal C}}

\def\ebf{\mathbf{e}}

\def\raw{\rightarrow}
\def\Pr{\textup{P}}
\DeclareMathOperator{\E}{E}

\def\be{\begin{equation}}
\def\ee{\end{equation}}
\def\bea{\begin{equation*}}
\def\eea{\end{equation*}}
\def\begs{\begin{split}}
\def\ends{\end{split}}

\newtheorem{thm}{Theorem}
\newtheorem{lma}[thm]{Lemma}

\newtheorem{prop}[thm]{Proposition}

\theoremstyle{remark}
\newtheorem*{remark}{Remark}
\newtheorem{preex}[thm]{Example}

\theoremstyle{definition}
\newtheorem*{acknow}{Acknowledgements}
\newtheorem*{keywords}{Keywords}

\begin{document}

\title{Convergence towards an asymptotic shape in first-passage percolation on cone-like subgraphs of the integer lattice
}

\date{\today}

\author{Daniel Ahlberg\\ \emph{University of Gothenburg}\\
\emph{Chalmers University of Technology}}
\maketitle

\begin{abstract}
In first-passage percolation on the integer lattice, the Shape Theorem provides precise conditions for convergence of the set of sites reachable within a given time from the origin, once rescaled, to a compact and convex limiting shape.
Here, we address convergence towards an asymptotic shape for cone-like subgraphs of the $\Z^d$ lattice, where $d\ge2$. In particular, we identify the asymptotic shapes associated to these graphs as restrictions of the asymptotic shape of the lattice. Apart from providing necessary and sufficient conditions for $L^p$- and almost sure convergence towards this shape, we investigate also stronger notions such as complete convergence and stability with respect to a dynamically evolving environment.

\begin{keywords}
First-passage percolation, shape theorem, large deviations, dynamical stability.
\end{keywords}
\end{abstract}

\section{Introduction}

The study of spatial growth in random environments has been in progress for about half a century by now, and the particular model studied in this paper dates back to the work of Hammersley and Welsh~\cite{hamwel65}. They introduced a model known as \emph{first-passage percolation}, whose foremost characteristic feature is its subadditive behaviour. In order to understand the asymptotic behaviour of the process, Hammersley and Welsh were led to develop a theory for subadditive stochastic sequences. The development was continued by Kingman, and culminated with the formulation of his Subadditive Ergodic Theorem in~\cite{kingman68}. The general theory and toolbox for the study of subadditive processes is often restricted to rather strong regularity conditions that in particular include stationarity. Such conditions are often met for processes defined on a regular graph or in a homogeneous space, but are usually violated as soon as regularity is lost, or an inhomogeneous environment is considered. These considerations motivates the present study in which first-passage percolation on certain `cone-like' subgraphs of the $\Z^d$ nearest-neighbour lattice, for $d\ge2$, is considered.

One of the main achievements in first-passage percolation is known as the Shape Theorem, whose complete statement should be attributed to Richardson~\cite{richardson73}, Cox and Durrett~\cite{coxdur81}, and Kesten~\cite{kesten86}. The Shape Theorem gives a law of large numbers for the spatial growth of the first-passage process on the lattice $\Z^d$. When the process is restricted to a subgraph of the lattice it is \emph{a priori} not obvious if, or under which conditions, a similar result may hold. Via a comparison between the first-passage process on the whole lattice and on `cylinder-shaped' graphs, we here extend the Shape Theorem to encompass also `cone-like' subgraphs. We provide some necessary and sufficient conditions for the convergence towards an asymptotic shape to hold in probability, in $L^p$ and almost surely, but also to satisfy stronger notions such as complete convergence and dynamical stability. An observation that we emphasize is that the obtained limiting shape is a mere restriction of the asymptotic shape associated with the lattice itself. We further remark that other aspects of first-passage percolation, related to the existence and structure of finite and infinite geodesics, was recently studied by Auffinger, Damron and Hanson~\cite{aufdamhan13} for subgraphs of the $\Z^2$ lattice.

Considering first-passage percolation on subgraphs of the lattice is the first step towards an `inhomogeneous' version of first-passage percolation, in which disjoint regions of the lattice (such as left and right half-space) are assigned weights according to different distributions. In the case that one of these distributions is dominated by the other, the evolution of the first-passage process in the dominated region is described by the results obtained from this study. However, to describe the evolution in the remaining region, or in the more general case when there is no prescribed relation between the two weight distributions, turns out to be harder and will be further studied in a forthcoming paper~\cite{Adamsid13}. That paper relies of results developed here in order to obtain results under optimal conditions.
\\

A random notion of distance on a graph $\Gc=(\Vc_\Gc,\Ec_\Gc)$ may be obtained by assigning i.i.d.\ non-negative weights $\tau_e$ to its edges $e\in\Ec_\Gc$. Interpreting the weights as lengths, it is possible to define a pseudo-metric between vertices of $\Gc$ as the minimal weight sum among all paths connecting the two points. Here a \emph{path} refers to an alternating sequence of vertices and edges; $v_0,e_1,v_1,\ldots,e_n,v_n$, such that $v_k$ is a common endpoint of the edges $e_k$ and $e_{k+1}$. (We will repeatedly abuse notation and identify a path with its set of edges.) So, for any path $\Gamma$ let $T(\Gamma):=\sum_{e\in\Gamma}\tau_e$, and define the distance, or \emph{travel time}, between $u$ and $v$ in $\Vc_\Gc$ as
\bea
T_\Gc(u,v):=\inf\big\{T(\Gamma):\Gamma\text{ is a path from }u\text{ to }v\big\}.
\eea
In the case when $\Gc$ equals the $\Z^d$ lattice, the random pseudo-metric induces a semi-norm $\muZ(\cdot)$ on $\R^d$. For any sequence $(x_n)_{n\ge1}$ of lattice points where $x_n$ for all $n\ge1$ is within a fixed distance from $nx$
\be\label{eq:timeconstant}
\lim_{n\to\infty}\frac{T(0,x_n)}{n}=\muZ(x),\quad\text{in probability},
\ee
(Here and below we drop the subscript on $T$ when $\Gc$ equals the lattice $\Z^d$.) The existence of the limit in~\eqref{eq:timeconstant} without a moment condition was obtained in~\cite{coxdur81,kesten86}. The convergence is further known to hold almost surely and in $L^1$ given that $\E[Y]<\infty$, where $Y$ denotes the minimum of the weights associated to the $2d$ edges adjacent to the origin, i.e.,
$$
Y:=\min\{\tau_e:e\text{ adjacent to the origin}\}.
$$
The latter is a consequence of the Subadditive Ergodic Theorem of~\cite{kingman68} (see also~\cite{A13}).

The more comprehensive Shape Theorem addresses simultaneous convergence in all directions. Different versions thereof was obtained in~\cite{richardson73},~\cite{coxdur81} and~\cite{kesten86}, and provide necessary and sufficient conditions for convergence of the set $\Wc(t)$ of sites reachable within time $t$ from the origin, once rescaled by time, to the unit ball $\Wc^\muZ=\{x\in\R^d:\muZ(x)\le1\}$ expressed in $\muZ$. A more compact way to describe a result of this kind is in terms of a limit. On this form Cox and Durrett's version states that if $\E[Y^d]<\infty$, then
\be\label{eq:latticeST}
\limsup_{z\in\Z^d:\,|z|\to\infty}\frac{|T(0,z)-\muZ(z)|}{|z|}=0\quad\text{almost surely}.
\ee
That $\E[Y^d]<\infty$ is necessary is also well-known. 
There are two regimes --- $\muZ\equiv0$ and $\muZ(x)\neq0$ for every $x\neq0$ --- offering quite different behaviour. In the latter regime $\muZ$ is a proper norm, and~\eqref{eq:latticeST} is equivalent to the following: If $\E[Y^d]<\infty$, then for every $\eps>0$, almost surely,
\be\label{eq:latticeinclusion}
(1-\eps)\Wc^\muZ\subset\frac{1}{t}\Wc(t)\subset(1+\eps)\Wc^\muZ\quad\text{for all large enough }t.
\ee

We are in this study is concerned with first-passage percolation on cone-like subgraphs of the $\Z^d$ lattice, with the principal aim of examining the conditions under which a shape theorem can be established. A subgraph $\Gc$ of the $\Z^d$ lattice will be referred to as \emph{induced} by $V\subset\R^d$ if $\Vc_\Gc=V\cap\Z^d$ and any two vertices in $\Gc$ are connected by an edge if (and only if) they are in the lattice.
Let $B(y,r):=\{x\in\R^d:|x-y|\le r\}$ denote the closed Euclidean ball. The graphs that we will consider will be of the form $\bigcup_{a\ge0}B(au,\omega(a))$, where $u$ is some unit direction and $\omega:[0,\infty)\to[0,\infty)$ an increasing function. For simplicity, we will restrict our attention to affine functions, which is indeed the more interesting case.

\subsection{A shape theorem for subgraphs of the integer lattice}

We will throughout the paper let $\Hc=\Hc(u,c)$ denote the subgraph of the $\Z^d$ lattice induced by the set $\bigcup_{a\ge0}B(au,ca+4\sqrt{d})$, for some unit vector $u\in\R^d$ and constant $c>0$. The term $4\sqrt{d}$ is required to make sure that the resulting graph is sufficiently well-connected, and will be motivated in Lemma~\ref{lma:2dpaths} below. Under a minimal moment assumption we may prove a weak law for the convergence towards an asymptotic shape on $\Hc$. The weak law can be stated as if $\E[Y^p]<\infty$ for some $p>0$, then for every $\eps>0$
\be\label{eq:weaklaw}
\limsup_{z\in\VH:\,|z|\to\infty}\Pr\big(|T_\Hc(0,z)-\muZ(z)|>\eps|z|\big)=0.
\ee
(See Proposition~\ref{prop:weakST} in Section~\ref{sec:discussion}.) However, under a stronger moment assumption we may obtain stronger modes of convergence. In order to present sharp conditions it will be necessary to separate the boundary from the bulk. Let $\Bc(u,c):=\bigcup_{a\ge0}B(au,ca)$, and partition the set $\VH$ of vertices of $\Hc(u,c)$ into its \emph{interior} $\iVH:=\Bc(u,c)\cap\Z^d$ and its \emph{boundary} $\bVH:=\VH\setminus\iVH$. Our main achievement is at occasions referred to as a Hsu-Robbins-Erd\H{o}s type of strong law, and characterizes the regime for polynomial rate of decay in~\eqref{eq:weaklaw}.

\begin{thm}\label{thm:HREcone}
Let $d\ge2$ and consider first-passage percolation on $\Hc=\Hc(u,c)$ for some unit vector $u\in\R^d$ and $c>0$. For every $\eps>0$ and $p>0$
$$
\E[Y^p]<\infty\quad\Leftrightarrow\quad\sum_{z\in\iVH}|z|^{p-d}\,\Pr\big(|T_\Hc(0,z)-\muZ(z)|>\eps|z|\big)<\infty.
$$
\end{thm}

In the case $\E[Y^d]<\infty$, then Theorem~\ref{thm:HREcone} implies that the cardinality of the set of sites $\big\{z\in\iVH:|T_\Hc(0,z)-\muZ(z)|>\eps|z|\big\}$ has finite expectation. This was by Hsu and Robbins~\cite{hsurob47}, in the context of random sequences, referred to as \emph{complete} convergence. Complete convergence implies almost sure convergence via Borel-Cantelli's lemma. Hence, as a corollary we obtain an analogue to~\eqref{eq:latticeinclusion} for first-passage percolation on $\Hc$. Let
$$
\Wc_\Hc(t):=\Big\{x\in\Bc(u,c):\exists\,z\in\VH\text{ for which }\|x-z\|_\infty\le1/2\text{ and }T_\Hc(0,z)\le t\Big\}.
$$
(Here $\|\cdot\|_\infty$ denotes the supremum norm on $\R^d$.) Assume again that $\muZ\not\equiv0$. Then, $\E[Y^d]<\infty$ implies that for every $\eps>0$, almost surely,
\be\label{eq:shapeinclusion}
(1-\eps)\Wc_\Hc^\muZ\subset\frac{1}{t}\Wc_t\subset(1+\eps)\Wc_\Hc^\muZ\quad\text{for all large enough }t,
\ee
where $\Wc_\Hc^\muZ:=\{x\in\R^d:\muZ(x)\le1\}\cap\Bc(u,c)$ equals the restriction to $\Bc(u,c)$ of the asymptotic shape $\Wc^\muZ$ appearing as the limit for first-passage percolation on $\Z^d$.

A Hsu-Robbins-Erd\H{o}s strong law analogous to Theorem~\ref{thm:HREcone} has in parallel be derived for the full $\Z^d$ lattice by the same author in~\cite{A13}. Just as in that study, the main ingredient in the proof will be an estimate on linear order deviations of travel times away from the time constant (Theorem~\ref{thm:LDEcone} below). Also based on this large deviations estimate, we further provide a complementary characterization of $L^p$-convergence.

\begin{thm}\label{thm:Lp-ST}
Let $d\ge2$ and consider first-passage percolation on $\Hc=\Hc(u,c)$ for some unit vector $u\in\R^d$ and $c>0$. For every $p>0$
$$
\E[Y^p]<\infty\quad\Leftrightarrow\quad\limsup_{z\in\iVH:\,|z|\to\infty}\E\left|\frac{T_\Hc(0,z)-\muZ(z)}{|z|}\right|^p=0.
$$
\end{thm}

The summation and limit in Theorem~\ref{thm:HREcone} and~\ref{thm:Lp-ST} run over interior points of $\Hc=\Hc(u,c)$. Note that if $c>1$, then $\Hc$ equals the lattice and the boundary is empty. On the other hand, if $c=1$ then $\Hc$ equals a half-space, and for $c\in(0,1)$ the set of interior point is simply those contained in a $d$-dimensional cone with radius $\frac{c}{\sqrt{1-c^2}}\cdot a$ at distance $a$ from its apex.
Extending the summation and the limit to include also points in the boundary is possible, but will have to be done to the cost of a somewhat stronger moment condition on the weight distribution. The reason for this, as well as some necessary and sufficient conditions for the extension of Theorem~\ref{thm:HREcone} and~\ref{thm:Lp-ST} to boundary points, will be discussed in Section~\ref{sec:discussion}.

\subsection{Super-linear growth outside the percolation regime}

As previously mentioned, $\muZ:\R^d\to[0,\infty)$ inherits the properties of a semi-norm. The precise condition for $\muZ$ to be an actual norm is closely related to the critical threshold $p_c(d)$ for bond percolation on $\Z^d$. Kesten~\cite{kesten86} showed that $\Pr(\tau_e=0)<p_c(d)$ is equivalent to $\mu(x)\neq0$ for all $x\neq0$. This may give the false impression that the condition for first-passage percolation on a graph $\Gc$ to obey a linear growth rate (in the sense of~\eqref{eq:shapeinclusion}), and thus not evolve in a more explosion-like manner, as a general fact coincides with the subcritical regime of percolation of 0-weight edges. To give an example where this does not happen, consider the subgraph of the $\Z^2$ lattice induced by the set $\big\{(x,y)\in\R^2:0\le y\le a\log(1+x)\big\}$, for some $a\in\R_+$. Grimmett~\cite{grimmett83} proved that for every $a\in\R_+$ the critical probability for bond percolation on this graph lies strictly between 1/2 and 1. Although this graph is not directly covered by Theorem~\ref{thm:HREcone}, one may by the same means conclude that travel times in the coordinate direction obey a similar strong law also for this graph, with its rate given by $\mu(\ebf_1)$, where $\ebf_i$ denotes the unit vector in the $i$th coordinate axis for $i=1,2,\ldots,d$. Since $\mu(\ebf_1)>0$ only if $\Pr(\tau_e=0)<1/2$, we arrive at the following conclusion.

\begin{prop}\label{prop:regime}
For every $a>0$, the critical probability for bond percolation on the subgraph $\Gc$ of the square lattice induced by $\big\{(x,y)\in\R^2:0\le y\le a\log(1+x)\big\}$ is strictly between $1/2$ and $1$. However, for any weight distribution with $\E[Y]<\infty$ that assigns at least mass $1/2$ to the value $0$, we have
$$
\lim_{n\to\infty}\frac{T_\Gc(0,n\ebf_1+\ebf_2)}{n}=0,\quad\text{almost surely}.
$$
\end{prop}

\subsection{Dynamical first-passage percolation}

In order to study the stability of travel times with respect to small perturbations, we introduce a dynamical version of first-passage percolation. The dynamical version is obtained by, apart from assigning i.i.d.\ weights, associating independent Poisson clocks to the edges of your graph. When a clock rings, the weight of the corresponding edge is resampled from the same distribution. Before formalizing the above dynamics further, let us mention its source of inspiration.

The dynamical first-passage model is inspired by the dynamical (bond) percolation model in which edges of a graph flip between `present' and `absent' at the toll of i.i.d.\ Poisson clocks. At each fixed time, an infinite open component exists with probability either 0 or 1 (depending on the probability of an edge being present). Is this property dynamically stable in the sense that for almost every realization we will see an infinite open component either present or absent at all times? This question was first studied by H\"aggstr\"om, Peres and Steif in \cite{hagperste97}. In a related work, Benjamini, H\"aggstr\"om, Peres and Steif \cite{benhagperste03} consider similar questions in the context of i.i.d.\ sequences.
The existence of exceptional times and the related concept of noise sensitivity was more profoundly studied for bond percolation on the square lattice, and site percolation on the triangular lattice, in an influential series of works by Benjamini, Kalai and Schramm \cite{BKS99}; Schramm and Steif \cite{schste10}; and Garban, Pete and Schramm \cite{garpetsch10}.

To give a more formal definition of the dynamical first-passage model, let $u\in\R^d$ have unit length and $c>0$, and associate independently to the edges of $\Hc=\Hc(u,c)$ a random process $\{\tau_e(s)\}_{s\ge0}$ defined as follows. Given non-negative i.i.d.\ random variables $\{\tau_e^{(j)}\}_{j\ge1}$ and independent rate-1 exponentially distributed random variables $\{\xi_e^{(j)}\}_{j\ge1}$, let $\xi_e^{(0)}=0$ and define for all $s\ge0$
$$
\tau_e(s):=\tau_e^{(j)},\quad\text{for }\sum_{k=0}^{j-1}\xi_e^{(k)}\,\le\, s\,<\,\sum_{k=0}^j\xi_e^{(k)}.
$$
Note that if $\tau_e^{(j)}$ has probability measure $\nu$, then for every $s\geq0$, the distribution of $\{\tau_e(s)\}_{e\in\Ec_\Hc}$ is given by the product measure $\nu^{\Ec_\Hc}$. Let $T_\Hc^{(s)}(y,z)$ denotes the travel time between $y$ and $z$ in $\VH$ with respect to $\{\tau_e(s)\}_{e\in\Ec_\Hc}$. Then
$$
T_\Hc^{(0)}(y,z)\,\stackrel{d}{=}\,T_\Hc^{(s)}(y,z),\quad\text{for all }s\geq0.
$$

It is natural to think of the dynamical environment as evolving with time. This gives us two time dimensions. It may help to interpret $T_\Hc^{(s)}(y,z)$ as the minimal \emph{cost} to travel between $y$ and $z$ at \emph{time} $s$, and $\big\{T_\Hc^{(s)}(y,z)\big\}_{s\ge0}$ as the evolution over time of the cost for that travel.

Combining Theorem~\ref{thm:HREcone} and Fubini's theorem shows that if $\E[Y^d]<\infty$, then
\be\label{eq:Fubini}
\Pr\left(\limsup_{z\in\iVH:\;|z|\to\infty}\frac{|T_\Hc^{(s)}(0,z)-\muZ(z)|}{|z|}=0\text{ for Lebesgue-a.e. }s\ge0\right)=1.
\ee
Our next result states that this can be strengthened to hold for \emph{every} $s\ge0$. The almost sure convergence towards an asymptotic shape is so to say \emph{dynamically stable} with respect to the introduced dynamics. The statement we shall prove is in fact stronger than so.

\begin{thm}\label{thm:DST}
Let $d\ge2$ and consider dynamical first-passage percolation on $\Hc=\Hc(u,c)$ for some unit vector $u\in\R^d$ and $c>0$. For every $\eps>0$ and $p>0$
$$
\E[Y^p]<\infty\quad\Leftrightarrow\quad\sum_{z\in\iVH}|z|^{p-d}\,\Pr\bigg(\sup_{s\in[0,1]}|T^{(s)}_\Hc(0,z)-\muZ(z)|>\eps|z|\bigg)<\infty.
$$
\end{thm}

When $\E[Y^d]$ is finite Theorem~\ref{thm:DST} shows that the expected number of sites $z\in\iVH$ whose travel time deviates by as much as $\eps|z|$ away from $\muZ(z)$ at some time during the interval $[0,1]$ is finite. As a consequence,~\eqref{eq:Fubini} holds `for every $s\ge0$', and in the regime that $\muZ\not\equiv0$, the dynamical analogue of~\eqref{eq:shapeinclusion} holds at all times $s\ge0$ with probability one.

\subsection{Notation}

We remind the reader that $|\cdot|$ will refer to Euclidean distance. We will further use $\|\cdot\|$ to denote $\ell^1$-distances on $\Z^d$. The unit vectors in the $d$ coordinate directions will be denoted by $\ebf_1,\ebf_2,\ldots,\ebf_d$, and we finally recall that $Y$ will denote the minimum weight among the $2d$ edges adjacent to the origin.

\subsection{Outline of the paper}

The rest of this paper is organized as follows. In the next section we provide some preliminary results concerning the time constant and the geometry of the lattice. We thereafter move on and derive a large deviations estimate for travel times on cones in Section~\ref{sec:LDE}. This result will later be the key to prove both Theorem~\ref{thm:HREcone} and~\ref{thm:Lp-ST} (in Section~\ref{sec:mainproofs}), as well as Theorem~\ref{thm:DST} (in Section~\ref{sec:dfpp}). A formal proof of Proposition~\ref{prop:regime} is omitted, but see however the closing remarks in Section~\ref{sec:LDE} and~\ref{sec:mainproofs}. The paper is finally ended in Section~\ref{sec:discussion} with a further discussion of some necessary and sufficient conditions under which Theorem~\ref{thm:HREcone},~\ref{thm:Lp-ST} and~\ref{thm:DST} can be extended to include boundary points.

\section{Preliminary observations}\label{sec:preliminaries}

A very central concept in first-passage percolation is the subadditive property, which says that
\bea
T(x,y)\,\le\, T(x,z)+T(z,y),\quad\text{for every }x,y,z\in\Z^d.
\eea
Its importance was recognized already in~\cite{hamwel65}, and was the inspiration for Kingman~\cite{kingman68} to derive his Subadditive Ergodic Theorem, from which the almost sure existence of the limit of $\lim_{n\to\infty}T(0,nz)/n$ can be deduced. Subadditivity of travel times carries over to the time constant. Together with the existence of the limit in~\eqref{eq:timeconstant}, one can show that the time constant satisfies the properties of a semi-norm. Clearly $\muZ(0)=0$, but also
\bea
\begin{aligned}
&\muZ(ax)\,=\,|a|\muZ(x) & \text{for all }a\in\R\text{ and }x\in\R^d,\\
&\muZ(x+y)\,\le\,\muZ(x)+\muZ(y) & \text{for all }x,y\in\R^d,\\
&\big|\muZ(y)-\muZ(x)\big|\,\le\,\muZ(\ebf_1)\|y-x\| & \text{for all }x,y\in\R^d.
\end{aligned}
\eea
(See~\cite{kesten86} for details.) Note that the final property follows from the previous two, and shows that $\muZ:\R^d\to[0,\infty)$ is Lipschitz continuous. The time constant is further known to be continuous with respect to weak convergence of the weight distribution. Let $F_1,F_2,\ldots,F_\infty$ denote distribution functions for the weights, and let $\muZ^F$ denote the time constant for the $\Z^d$ lattice associated with the distribution $F$. Then, if $F_n\to F_\infty$ weakly,
\be\label{eq:mucontinuity}
\lim_{n\to\infty}\muZ^{F_n}(x)=\muZ^{F_\infty}(x),\quad\text{uniformly in $x$ on compact sets}.
\ee
This convergence was first proved for $d=2$ and $x=\ebf_1$, by Cox and Kesten \cite{cox80,coxkes81}, and later extended to all directions and $d\ge2$ \cite{kesten86}. That the convergence is uniform on compact sets follows easily from the Lipschitz continuity of $\mu$: Given $\eps>0$, let $K\subset\R^d$ be compact and let $K'$ be an $\eps$-dense finite subset of $K$ (each element of $K$ is at most at $\ell^1$-distance $\eps$ from an element in $K'$). Pick $N<\infty$ such that $\big|\mu^{F_n}(x)-\mu^F(x)\big|<\eps$ for all $x\in K'\cup\{\ebf_1\}$ and $n\ge N$. Lipschitz continuity of $\mu$ then gives that $\big|\mu^{F_n}(x)-\mu^F(x)\big|<(2\mu^F(\ebf_1)+2)\eps$ for all $x\in K$ and $n\ge N$. Since $\eps>0$ was arbitrary, this proves the claim.

\subsection{Lattice geometry}\label{sec:latticegeo}

In order for present results for first-passage percolation on subgraphs of comparable strength to those of the $\Z^d$ lattice, the subgraphs under considerations will need to have a similar degree of connectivity. We will explore some sufficient conditions for a `good' degree of connectivity in a couple of lemmas.

\begin{lma}\label{lma:connected}
For $z\in\Z^d$ and $r\ge\sqrt{d}$, the subgraph of $\Z^d$ induced by $\bigcup_{a\in[0,1]}B(az,r)$ is connected.
\end{lma}

\begin{proof}
Let $U$ denote the union of all unit cubes whose corners have integer coordinates, and that intersect the straight line segment between the origin and $z$. Since $r\ge\sqrt{d}$, it follows that $U\subset\bigcup_{a\in[0,1]}B(az,r)$, and hence that all corners of a cube in $U$ are connected in the induced graph. Since each ball $B(az,r)$ is convex and $az$ is contained in some unit cube that $U$ is made up by, it follows that the graph is connected.
\end{proof}

The same idea can be used to give a lower bound on $r$ for the origin and $z$ to be connected by $2d$ disjoint paths. A somewhat more general statement is given next.

\begin{lma}\label{lma:2dpaths}
Let $u\in\R^d$ have unit length, $r\ge4\sqrt{d}$ and $0\le b\le c$ be real numbers. Between any pair of points $y$ and $z$ in $\Z^d$ at distance at most $\sqrt{d}$ of the line segment $\{au:a\in[b,c]\}$ there is a path of length $\|z-y\|$, and between each consecutive pair of points in the path there are $2d$ edge-disjoint paths of length $9$, all contained in $\bigcup_{a\in[b,c]}B(au,r)$. 
\end{lma}

\begin{proof}
Let $U$ denote the union of all unit cubes whose corners have integer coordinates, and that intersect the straight line segment between $y$ and $z$. A path between $y$ and $z$ of length $\|z-y\|$ can now be chosen by walking along edges of the unit cubes that $U$ is made up by, in each step decreasing the $\ell^1$-distance to its target by one. Since no point on the line segment between $y$ and $z$ lies further than $\sqrt{d}$ from a point on the line segment $\{au:a\in[b,c]\}$, no point on the path may lie further than distance $2\sqrt{d}$ from $\{au:a\in[b,c]\}$.

Note that there are $2d$ paths of length at most $9$ between the origin and $\ebf_1$ that remain within distance $\sqrt{5}$ of either the origin or $\ebf_1$. The analogous statement holds for each pair of consecutive points on the path between $y$ and $z$. Since each such point is at distance at most $2\sqrt{d}$ of the line segment $\{au:a\in[b,c]\}$, all $2d$ paths between any pair of consecutive points on the path between $y$ and $z$ may lie further than $2\sqrt{d}+\sqrt{5}$ from $\{au:a\in[b,c]\}$.
\end{proof}

Recall that $Y$ denotes the minimum of $2d$ independent random variables distributed as $\tau_e$. The latter lemma allows us to compare the probability tail of the travel time between points in a thin cylinder with that of $Y$.

\begin{prop}\label{prop:Ycond}
Let $b\le c$ and $r\ge4\sqrt{d}$ be real numbers, and $u$ a unit vector in $\R^d$. Let $\Gc$ denote the subgraph of the $\Z^d$ lattice induced by $\bigcup_{a\in[b,c]}B(au,r)$. For any pair of lattice points $y$ and $z$ at distance at most $\sqrt{d}$ from the line segment $\{au:a\in[b,c]\}$ we have
$$
\Pr\big(T_\Gc(y,z)>9\|z-y\|x\big)\,\le\, 9^{2d}\,\|z-y\|\,\Pr(Y>x).
$$
\end{prop}

\begin{proof}
Let $v_0=y$, $v_{\|z-y\|}=z$, and let $v_1,v_2,\ldots,v_{\|z-y\|-1}$ denote the intermediate points on the path between $y$ and $z$ mentioned in Lemma~\ref{lma:2dpaths}. Each consecutive pair of points is joint by $2d$ edge disjoint paths of length at most 9. Denote these paths by $\Gamma_1,\Gamma_2,\ldots,\Gamma_{2d}$, and assume that $\Gamma_1$ is the longest among them. Clearly
$$
\Pr\Big(\min_{i=1,2,\ldots,2d}T(\Gamma_i)>9x\Big)\;\le\;\Pr\big(T(\Gamma_1)>9x\big)^{2d}\;\le\;9^{2d}\,\Pr(\tau_e>x)^{2d}.
$$
Next, note that $T_\Gc(y,z)$ is dominated by the sum of $T_\Gc(v_{i-1},v_i)$ for $i=1,2,\ldots,\|z-y\|$. However, each term in the sum is dominated by a random variables distributed as $\min_{i=1,2,\ldots,2d}T(\Gamma_i)$. Consequently,
\bea
\Pr\big(T_\Gc(y,z)>9\|z-y\|x\big)\;\le\;\|z-y\|\,\Pr\Big(\min_{i=1,2,\ldots,2d}T(\Gamma_i)>9x\Big)\;\le\;9^{2d}\,\|z-y\|\,\Pr(Y>x).\qedhere
\eea
\end{proof}

\subsection{Comparison between lattice and subgraphs}

After we in the above subsection have identified some simple restrictions to obtain optimal necessary and sufficient conditions for the main results of this paper, let us continue with an indication to why growth on cone-like subgraphs of the lattice has the same asymptotic rate of growth that the lattice itself. To succeed with such a quantitative comparison, we will naturally need to compare travel times on the lattice with those of a subgraph. An efficient way to handle that is to let $\{\tau_e\}_{e\in\Ec_{\Z^d}}$ be an i.i.d.\ family of edge weights associated with the edges of the $\Z^d$ lattice, and for each subgraphs $\Gc$ of the lattice $\Z^d$ let distances in $\Gc$ be determined in terms of the restriction of $\{\tau_e\}_{e\in\Ec_{\Z^d}}$ to $\Ec_\Gc$. That is, we will let $T_\Gc(\cdot,\cdot)$ denote the infimum over all weight sums with respect to $\{\tau_e\}_{e\in\Ec_{\Z^d}}$, over paths contained in $\Gc$. This convention generates a natural coupling of travel times among subgraphs of the lattice.
 Indeed if $\Gc_1$ is a subgraphs of the lattice, and $\Gc_2$ is a further subgraph of $\Gc_1$, then
$$
T_{\Gc_1}(y,z)\,\le\,T_{\Gc_2}(y,z)\quad\text{for all }y,z\in\Gc_2.
$$
We will from now on assume this coupling to be in force. As before, once $\Gc$ equals $\Z^d$ we allow ourselves to drop the subscript.

Let us next give the first hint to why travel times on a cone-like subgraph of the lattice has the same linear growth rate as travel times on the whole lattice. Given $z\in\Z^d$ and $r\ge0$, let $\Cyl(z,r)$ denote the subgraph of the lattice induced by the cylinder $\bigcup_{a\in\R}B(az,r)$. First-passage percolation on such (and more general) 1-dimensional graphs was studied in detail in~\cite{A11-1}. In particular it was proven that if $\E[Y]<\infty$ (and $r\ge4\sqrt{d}$, see Lemma~\ref{lma:2dpaths}), then $T_{\Cyl(z,r)}(0,nz)/n$ converges almost surely as $n\to\infty$ to a limiting value $\mu_{\Cyl(z,r)}$. That $\mu_{\Cyl(z,r)}$ is non-increasing in $r$ is an immediate consequence of the above coupling. It is further bounded from below by $\muZ(z)$, and hence convergent as $r$ tends to infinity. The limit indeed equals $\muZ(z)$.

\begin{prop}\label{prop:muKconv}
Assume that $\E[Y]<\infty$. For every $z\in\Z^d$
$$
\mu_{\Cyl(z,r)}\to\muZ(z)\quad\text{as }r\to\infty.
$$
\end{prop}

\begin{proof}
Clearly $T_{\Cyl(z,r)}(0,nz)$ is decreasing in $r$. Hence, for every $n$ we get
$$
T(0,nz)=\lim_{r\raw\infty}T_{\Cyl(z,r)}(0,nz)=\inf_{r\ge0}T_{\Cyl(z,r)}(0,nz),\quad\text{almost surely}.
$$
An application of the Monotone Convergence Theorem gives
$$
\E\big[T(0,nz)\big]\;=\;\lim_{r\raw\infty}\E\big[T_{\Cyl(z,r)}(0,nz)\big]\;=\;\inf_{r\ge0}\E\big[T_{\Cyl(z,r)}(0,nz)\big].
$$
By Fekete's lemma the limit of $a_n/n$ as $n\to\infty$ exists and equals $\inf_{n\ge1}a_n/n$, for any subadditive real-valued sequence $(a_n)_{n\ge1}$.
Thus, since $\mu_{\Cyl(z,r)}$ is non-increasing in $r$ we conclude that
\bea
\begin{split}
\lim_{r\raw\infty}\mu_{\Cyl(z,r)}\;&=\;\inf_{r\ge0}\inf_{n\ge1}\frac{\E\big[T_{\Cyl(z,r)}(0,nz)\big]}{n}\;=\;\inf_{n\ge1}\inf_{r\ge0}\frac{\E\big[T_{\Cyl(z,r)}(0,nz)\big]}{n}\\
&=\;\inf_{n\ge1}\frac{\E\big[T(0,nz)\big]}{n}\;=\;\muZ(z).\qedhere
\end{split}
\eea
\end{proof}

\begin{remark}
Importantly, it suffices that $\E[Y^p]<\infty$ for some $p>0$ to be able to define a time constant for $\Cyl(z,r)$, when $r$ is large enough. Also under the relaxed condition the limit of $\mu_{\Cyl(z,r)}$ as $r\to\infty$ equals $\muZ(z)$. The complete argument is presented in~\cite{A13}.
\end{remark}

\section{A large deviations estimate for cones}\label{sec:LDE}

We will in this section establish a direct relation between large deviations and the probability tails of the random variable $Y$. We will in the following couple of sections use this large deviations estimate to prove Theorem~\ref{thm:HREcone},~\ref{thm:Lp-ST} and~\ref{thm:DST}. Recall that $\Hc=\Hc(u,c)$ denotes the subgraph of the $\Z^d$ lattice induced by $\bigcup_{a\ge0}B(au,ca+4\sqrt{d})$, for some unit vectors $u$ and constant $c>0$. Since $T_{\Hc(u,c)}(0,z)$ is at least as large as the minimum among the weights assigned to the $2d$ edges adjacent to the origin, it easily follows that for every $x\ge\|z\|$
$$
\Pr\big(|T_{\Hc(u,c)}(0,z)-\muZ(z)|>\eps x\big)\,\ge\,\Pr\big(Y>(\muZ(\ebf_1)+1)x\big).
$$
The result we shall prove here provides a matching upper bound.

\begin{thm}\label{thm:LDEcone}
Assume that $\E[Y^p]<\infty$ for some $p>0$. For every $\eps>0$, unit vector $u\in\R^d$, $c>0$ and $q\ge1$ there exists $M=M(\eps,d,p,q,u)$ such that for every $z\in\iVH$ and $x\ge\|z\|$
 $$
 \Pr\big(|T_{\Hc(u,c)}(0,z)-\muZ(z)|>\eps x\big)\,\le\,M\,\Pr(Y>x/M)+\frac{M}{x^q}.
 $$
\end{thm}

A proof of the corresponding result in the case that $\Hc$ equals the $\Z^d$ lattice is presented in parallel in~\cite{A13}. We here focus on how to extend that result to $\Hc(u,c)$. Note that deviations below and above the time constant may easily be treated separately via the identity
$$
\Pr\big(|T_\Hc(0,z)-\muZ(z)|>\eps x\big)\,=\,\Pr\big(T_\Hc(0,z)-\muZ(z)<-\eps x/2\big)+\Pr\big(T_\Hc(0,z)-\muZ(z)>\eps x/2\big).
$$
Since travel times on $\Hc$ are bounded from below by travel times on the lattice, we will be able to use the fact (from~\cite{A13}) that Theorem~\ref{thm:LDEcone} holds for the $\Z^d$ lattice to obtain an estimate on deviations below the time constant. What we instead need to do here is to provide an estimate for deviations above the time constant.

\subsection{Deviations above the time constant}

We will follow the approach used in~\cite{A13}, and work with cylinder-shaped graphs. Let $\Cyl(z,r)$, for some $z\in\Z^d$ and $r\ge\sqrt{d}$ (cf.\ Lemma~\ref{lma:connected}), denote the subgraph of the $\Z^d$ lattice induced by the infinite cylinder $\bigcup_{a\in\R^d}B(az,r)$. Let also $\Hb_n:=\{z\in\Z^d:z_1+z_2+\ldots+z_d=n\}$. The first step in the proof of Theorem~\ref{thm:LDEcone} is to obtain an estimate for deviations of travel times along cylinders. For ease of presentation we will state our first lemma for $z$ belonging to the first orthant, and note that the analogous statement holds for all $z\in\Z^d$ due to symmetry.

\begin{lma}[\cite{A13}]\label{lma:tube}
 Assume $\E[Y^p]<\infty$ for some $p>0$. For every $\eps>0$, $q\ge1$ and $z\in\N^d$, there exists $R_1=R_1(d,p,q)$ and $M_1=M_1(\eps,d,p,q,z)$ such that for every $r\ge R_1$, $n\in\N$ and $x\ge n$
 $$
 \Pr\big(T_{\Cyl(z,r)}(\Hb_0,\Hb_{n\|z\|})-n\muZ(z)>\eps x\|z\|\big)\,\le\,\frac{M_1}{x^q}.
 $$
\end{lma}

Since Lemma~\ref{lma:tube} is proved in~\cite{A13}, we only recall the main steps here.

\begin{proof}[Outline of proof]
The key to the proof lies in identifying a suitable regenerative sequence. First fix $\eps>0$, and choose $r$ large enough for $T_{\Cyl(z,r)}(\Hb_0,\Hb_{\|z\|})$ to have finite variance and for $\mu_{\Cyl(z,r)}$ to be close to $\muZ(z)$ (see Proposition~\ref{prop:muKconv} and the following remark). Next choose $m$ large for $\frac{1}{m}\E\big[T_{\Cyl(z,r)}(\Hb_0,\Hb_{m\|z\|})\big]$ to be close to $\mu_{\Cyl(z,r)}$.

An easy way to identify a sequence which is almost regenerative is to, for some $t>0$, let $A_k$ denote the event that all edges connecting $\Hb_{km\|z\|-1}$ to $\Hb_{km\|z\|}$ and contained in $\Cyl(z,r)$ have weight at most $t$. Let $\rho_j$ denote the $j$th $k\ge1$ for which $A_k$ occurs. Then the sequence $\big(T_{\Cyl(z,r)}(\Hb_0,\Hb_{\rho_jm\|z\|})\big)_{j\ge1}$ is not regenerative, but $T_{\Cyl(z,r)}(\Hb_0,\Hb_{\rho_jm\|z\|})$ is well approximated by the sum of $T_{\Cyl(z,r)}(\Hb_{\rho_{i-1}m\|z\|},\Hb_{\rho_im\|z\|})$ for $i=1,2,\ldots,j$ (where $\rho_0=0$). Note further that these terms are i.i.d.\ by construction, and have finite variance.

For the next step, let $\nu(n):=\min\{j\ge1:\rho_jm>n\}$ and note that $\nu(n)$ is a stopping time. The stopped sum $\sum_{i=1}^{\nu(n)}T_{\Cyl(z,r)}(\Hb_{\rho_{i-1}m\|z\|},\Hb_{\rho_im\|z\|})$ approximates $T_{\Cyl(z,r)}(\Hb_0,\Hb_{n\|z\|})$, and one may show that its mean is close $n\muZ(z)$. Together with an estimate on the error of the approximation it is possible to apply Chebychev's inequality and Wald's lemma to obtain polynomial decay of the probability $\Pr\big(T_{\Cyl(z,r)}(\Hb_0,\Hb_{n\|z\|})-n\muZ(z)>\eps x\|z\|\big)$. Once polynomial decay is achieved, it can be strengthened to a higher power by considering disjoint cylinders aligned next to each other, all of radius $r$.

The missing details that justify all steps above are fully presented in~\cite[see Proposition~18 and its proof]{A13}.
\end{proof}

To obtain an upper bound on travel times on a cone in terms of travel times on a cylinder, it will be necessary to make sure the cylinder is contained in the cone. We will therefore need a refined version of the above lemma. Let $\Dc(x,y,r)$ denote the subgraph of the lattice induced by $\bigcup_{a\in[0,1]}B(x+a(y-x),r)$, for some $x,y\in\R^d$ and $r>4\sqrt{d}$. That is, $\Dc(x,y,r)$ is a cylinder with round ends at its `endpoints' $x$ and $y$.

\begin{prop}\label{prop:tube}
Assume that $\E[Y^p]<\infty$ for some $p>0$. For every $\eps>0$, $q\ge1$ and $z\in\Z^d$, there exist $R_2=R_2(d,p,q)$ and $M_2=M_2(\eps,d,p,q,z)$ such that for every $a,b\in\R$, $r\ge R_2$, $u,v\in\Dc(az,bz,\sqrt{d})$ and $x\ge\|v-u\|$
$$
\Pr\big(T_{\Dc(az,bz,r)}(u,v)-\|v-u\|\muZ(z/\|z\|)>\eps x\big)\,\le\, M_2\,\Pr(Y>x/M_2)+\frac{M_2}{x^q}.
$$
\end{prop}

\begin{proof}
Due to symmetry we may assume that $z$ belongs to the first orthant. Let $R_1=R_1(d,p,q)$  and $M_1=M_1(\eps,d,p,q,z)$ be given as in Lemma~\ref{lma:tube}. We will first show that
for every $r\ge R_1$, $m,n\in\N$ and large enough $x\ge n$
\be\label{eq:firstbound}
\Pr\big(T_{\Cyl(z,r)}(\Hb_m,\Hb_{m+n})-n\muZ(z/\|z\|)>\eps x\big)\,\le\,M_1\left(\frac{2\|z\|}{x}\right)^q.
\ee
Let $k_n$ denote the largest integer such that $k_n\|z\|\le n$. An upper bound on the left-hand side in~\eqref{eq:firstbound} is given by
$$
\Pr\Big(T_{\Cyl(z,r)}(\Hb_{k_m\|z\|},\Hb_{(k_{m+n}+1)\|z\|})-(k_{m+n}-k_m-1)\muZ(z)>\eps x\Big).
$$
According to Lemma~\ref{lma:tube}, this expression is bounded above by $M_1(2\|z\|/x)^q+\Pr(2\muZ(z)>\eps x/2)$ for all $x\ge(k_{m+n}-k_m+1)\|z\|$, and thus in particular for $x\ge n+2\|z\|$. Since $\eps>0$ was arbitrary, and since $\muZ(z)\le\eps x/4$ when $x$ is large, we obtain~\eqref{eq:firstbound}.

We continue with the main statement. Assume that $\|u\|\le\|v\|$, and pick $R_2'=R_2'(d,p,q)$ such that $\bigcup_{n=\|u\|,\ldots,\|v\|}\Hb_n\cap\Cyl(z,R_1)$ is contained in $\Dc(az,bz,R_2')$. (A specific value of $R_2'$ is not important, but one can show that $R_2'\ge\sqrt{d}R_1+\sqrt{d}$ is sufficient.) For ease of notation let $V_n=\Hb_n\cap\Cyl(z,R_1)$, and note that for $r\ge R_2=R_2'+4\sqrt{d}$
$$
T_{\Dc(az,bz,r)}(u,v)\,\le\,T_{\Cyl(z,R_1)}(\Hb_{\|u\|},\Hb_{\|v\|})+\sum_{y\in V_{\|u\|}}T_{\Dc(az,bz,r)}(u,y)+\sum_{y\in V_{\|v\|}}T_{\Dc(az,bz,r)}(y,v).
$$
Note further that $\big|\|v\|-\|u\|\big|\le\|v-u\|$, so by~\eqref{eq:firstbound} we obtain for large enough $x\ge\|v-u\|$
$$
\Pr\Big(T_{\Cyl(z,R)}(\Hb_{\|u\|},\Hb_{\|v\|})-\|v-u\|\muZ(z/\|z\|)>\eps x\Big)\,\le\,M_1\left(\frac{2\|z\|}{x}\right)^q.
$$
Since $\eps>0$ was arbitrary, the statement is now easily derived from Proposition~\ref{prop:Ycond}.
\end{proof}

A second difficulty in the extension of the large deviations estimate from lattice to cone comes with the next proposition. A version of the result was crucial also to derive the lattice estimate. The more restrictive geometry of the cone complicates the situation somewhat.

\begin{prop}\label{prop:cone}
Assume that $\E[Y^p]<\infty$ for some $p>0$. For every unit vector $u\in\R^d$, $c>0$ and $q\ge1$ there is $M_3=M_3(d,p,q,u)$ such that for every $y,z\in\iVH$ and $x\ge\|z-y\|$
$$
\Pr\big(T_{\Hc(u,c)}(y,z)>M_3x\big)\,\le\, M_3\,\Pr(Y>x)+\frac{1}{x^q}.
$$
\end{prop}

\begin{proof}
There are three cases, depending on whether $c>1$, $c=1$ or $c<1$. In the first case $\Hc$ equals the lattice, and the statement is essentially derived by obtaining an upper bound on the travel time between $y$ and $z$ as follows: Let $v^{(0)}=y$ and $v^{(d)}=z$, and define $d-1$ intermediate points as the turning points of the path between $y$ and $z$ that successively takes $z_i-y_i$ steps in direction $\ebf_i$, for $i=1,2,\ldots,d$. Denote the intermediate points by $v^{(1)},v^{(2)},\ldots,v^{(d-1)}$. An upper bound on the travel time between $y$ and $z$ is given by the sum of $T(v^{(i-1)},v^{(i)})$ for $i=1,2,\ldots,d$, and the result now follows via Proposition~\ref{prop:tube}.

In the second case $\Hc$ equals the restriction of the lattice to a half-space, and the above approach works also here, although one has to make sure that the intermediate points are chosen within the half-space. In the final case, we may argue similarly, although it will not suffice to take steps in coordinate directions. We will instead have to choose the $d$ base vectors more carefully.

Assume thus that $c<1$, and further that $u$ is of the form $z/|z|$ for some $z\in\Z^d$. We will at the end indicate how the proof may be modified to general directions. Let $H_u$ denote the hyperplane through the origin to which $u$ is normal, and let $H_u(z)$ denote its transposition to a point $z$. Note that the intersection of $H_u(z)$ with $\Hc=\Hc(u,c)$ is a $(d-1)$-dimensional Euclidean ball. Let $u^{(1)}=u$, and pick $d-1$ orthonormal vectors $u^{(2)},u^{(3)},\ldots,u^{(d)}$ spanning $H_u$, and also them `rational', i.e.\ of the form $z/|z|$ for some $z\in\Z^d$.

Pick $R_2=R_2(d,p,q)$ as in Proposition~\ref{prop:tube}, and assume first that $y$ and $z$ are at distance at least $R$ from the boundary of the cone, i.e.\ that $B(y,R_2)$ and $B(z,R_2)$ are both contained in $\VH$. We will next identify $d-1$ intermediate points, and bound the travel time between $y$ and $z$ by the sum of travel times between these intermediate points. We may assume that $y$ lies no further than $z$ from the origin. Set $v^{(0)}=y$, and let $v^{(1)}$ denote the intersection of $H_u(z)$ and the straight line through $y$ in direction $u$. (Note that $v^{(1)}$ does not necessarily have integer coordinates.) Next, move from $v^{(1)}$ to $z$ along $d-1$ straight line segments parallel to the vectors $u^{(2)},\ldots,u^{(d)}$. Since $y$ lies at distance at least $R_2$ from the boundary of $\Hc$, also $v^{(1)}$ will lie at distance at least $R_2$ from the boundary. Reordering the sequence in which the line segments are crossed, we can make sure that the endpoints of the line segments are all contained in the intersection between $H_n(z)$ and $\Hc$, and that none is closer than $R_2$ from the boundary of the cone. Denote these intermediate points connecting two line segments by $v^{(2)},\ldots,v^{(d-1)}$. Finally, let $v^{(d)}=z$.

Since the $d$ unit vectors are orthogonal to each other, the length of the straight line segments joining $v^{(0)},v^{(1)},\ldots,v^{(d)}$ together equals $\|z-y\|$. For $i=0,1,\ldots,d$, let $\hat{v}^{(i)}$ denote the lattice point closest to $v^{(i)}$. According to Proposition~\ref{prop:tube} the travel time between $\hat{v}^{(i-1)}$ and $\hat{v}^{(i)}$, restricted to $\Dc(v^{(i-1)},v^{(i)},R_2)$, satisfies for each $i=1,2,\ldots,d$
\be\label{eq:boundstep}
\Pr\Big(T_{\Dc(v^{(i-1)},v^{(i)},R_2)}(\hat{v}^{(i-1)},\hat{v}^{(i)})>(\muZ(u^{(i)})+1)x\Big)\,\le\, M_2^{(i)}\,\Pr\big(Y>x/M_2^{(i)}\big)+\frac{M_2^{(i)}}{x^q},
\ee
for some $M_2^{(i)}=M_2^{(i)}(d,p,q,u^{(i)})$ and $x\ge\|\hat{v}^{(i)}-\hat{v}^{(i-1)}\|$. Note further that since each of the points $v^{(0)},v^{(1)},\ldots,v^{(d)}$ is at distance at least $R_2$ from the boundary of $\Hc$, it follows that $\Dc(v^{(i-1)},v^{(i)},R_2)$ is contained in $\Hc=\Hc(u,c)$. Consequently,
$$
T_\Hc(y,z)\;\le\; \sum_{i=1}^dT_\Hc(\hat{v}^{(i-1)},\hat{v}^{(i)})\;\le\;\sum_{i=1}^d T_{\Dc(v^{(i-1)}v^{(i)},R_2)}(\hat{v}^{(i-1)},\hat{v}^{(i)}).
$$
Together with~\eqref{eq:boundstep} this implies that for `rational' $u$, and for $y$ and $z$ at distance at least $R_2$ from the boundary of $\Hc$ we have proven that for $x\ge\|y-z\|$
$$
\Pr\big(T_\Hc(y,z)>M_3'x\big)\,\le\, M_3'\,\Pr(Y>x/M_3')+\frac{M_3'}{x^q},
$$
for any $M_3'$ greater than both $M_2^{(1)}+\ldots+M_2^{(d)}$  and $\muZ(u^{(1)})+\ldots+\muZ(u^{(d)})+d$. The statement is now easy to extend to any pair of interior point via Proposition~\ref{prop:Ycond}.

It remains to argue how the proof may be modified in the case that $u$ is not `rational'. In this case, we may chose a unit vector $u'$ of the form $z/|z|$ for some $z\in\Z^d$ which is an arbitrarily close approximation of $u$, and a set of $d-1$ orthonormal vectors spanning $H_{u'}$. The intersection of $H_{u'}$ and the cone is in this case a solid $(d-1)$-dimensional ellipsoid, and once $u'$ is close to $u$, the ellipsoid is close to a sphere. It is possible to choose the $d$ vectors to arbitrarily well approximate $u$ and the $d-1$ principal directions of the ellipsoid obtained as the intersection between $H_{u'}$ and the cone. Although the approximations are well chosen, there may still be pairs of points within the ellipsoid which may not be connected by $d-1$ straight line segments (parallel with the $d-1$ spanning vectors) without exiting the ellipsoid. However this can certainly be achieved by $d+1$ line segments as follows: From each point in the solid ellipsoid it is possible to reach the interior 
of a large sphere contained in the ellipsoid following a straight line segment parallel to one of the $d-1$ chosen unit directions. Any two points within the sphere may as before be connected in $d-1$ steps, and this ends the modifications necessary in the case that $u$ is arbitrary.
\end{proof}

The proof of the main result of this section may now be given.

\begin{proof}[{\bf Proof of Theorem~\ref{thm:LDEcone}}]
Fix $\eps\in(0,1/2)$ and $q\ge1$. Recall that travel times on $\Hc=\Hc(u,c)$ are bounded below by travel times of the lattice, thus the inequality
$$
\Pr\big(|T_\Hc(0,z)-\muZ(z)|>\eps x\big)\,\le\,\Pr\big(T(0,z)-\muZ(z)<-\eps x/2\big)+\Pr\big(T_\Hc(0,z)-\muZ(z)>\eps x/2\big).
$$
Recall further that the correct order of decay of $\Pr\big(T(0,z)-\mu(z)<-\eps x\big)$ was obtained in~\cite{A13}. It will therefore suffice to show that there exists $M=M(\eps,d,p,q,u)$ such that for every $z\in\iVH$ and $x\ge\|z\|$
\be\label{eq:LDEabove}
\Pr\big(T_\Hc(0,z)-\muZ(z)>\eps x\big)\,\le\, M\,\Pr(Y>x/M)+\frac{M}{x^q}.
\ee
 
Let $R_2=R_2(d,p,q)$ be as in Proposition~\ref{prop:tube}. Let $U_n$ denote the set of points $y\in\Z^d$ such that $\|y\|=n$ and $B(y,R_2)\subset\VH$. Choose $N=N(\eps,R_2)$ such that all but at most finitely many points in $\VH$ are contained in
$$
\bigcup_{y\in U_N}\bigcup_{a\ge0}B(ay,\eps a|y|).
$$
That these finitely many points satisfies~\eqref{eq:LDEabove} for some $M$ is immediate from Proposition~\ref{prop:Ycond}. For any other other point $z\in\VH$, note that for some $y\in V_N$ (the one closest to $z$)
$$
\big\|y/\|y\|-z/\|z\|\big\|\;\le\;\big|y/|y|-z/|z|\big|\;\le\;2\eps.
$$

We proceed to obtain an upper bound on $T_\Hc(0,z)-\muZ(z)$. Write $\Dc$ for $\Dc(y,\|z\|y/\|y\|,R_2)$, and $y_z$ for the lattice point closest to $z$ among those at distance at most $\sqrt{d}$ from $\|z\|y/\|y\|$ and with $\|y_z\|\le\|z\|$. Note that also $\|y_z-y\|\le\|z\|$ since $\|y\|$ is assumed to be large. The usual subadditive property and coupling gives us
 \bea
 \begin{aligned}
 T_\Hc(0,z)-\muZ(z)\;&\le\;T_\Hc(0,y)+\big[T_\Dc(y,y_z)-\|y_z-y\|\muZ(y/\|y\|)\big]\\
 &\quad\,+\|z\|\big[\muZ(y/\|y\|)-\muZ(z/\|z\|)\big]+T_\Hc(y_z,z).
 \end{aligned}
 \eea
The four terms in the right-hand side, to be referred to as $X_1,X_2,X_3,X_4$, will be considered in order. Let $M_2=M_2(\eps,d,p,q,y)$ be as in Proposition~\ref{prop:tube}, let $M_3=M_3(d,p,q,u)$ be as in Proposition~\ref{prop:cone}, and let $x\ge\|z\|$. As before, Proposition~\ref{prop:Ycond} shows that $\Pr(X_1>\eps x)$ is bounded above by $9^{2d}\,\|y\|\,\Pr(Y>\eps x)$. Secondly, for $\Pr(X_2>\eps x)$ Proposition~\ref{prop:tube} gives the upper bound $M_2\,\Pr(Y>x/M_2)+M_2/x^q$. Due to Lipschitz continuity of the time constant, the third term is at most $\|z\|\muZ(\ebf_1)\big\|y/\|y\|-z/\|z\|\big\|$ which again is at most $2\muZ(\ebf_1)\eps\|z\|$. As a consequence $\Pr(X_3>2\muZ(\ebf_1)\eps x)=0$. Finally, note that $y_z$ is closer to $z$ than $\|z\|y/\|y\|$ is, so $\|y_z-z\|\le2\eps\|z\|$. Via Proposition~\ref{prop:cone} we thus obtain the upper bound $M_3\,\Pr(Y>x)+x^{-q}$ on $\Pr(X_4>2M_3\eps x)$.

In conclusion, for all large enough $z$ and $x\ge\|z\|$
 $$
 \Pr\Big(T_\Hc(0,z)-\muZ(z)>(1+1+2\muZ(\ebf_1)+2M_3)\eps x\Big)\,\le\, M\,\Pr(Y>x/M)+\frac{M}{x^q},
 $$
 where $M=9^{2d}\,\|y\|+1/\eps+M_2+M_3$ suffices. Since $\eps>0$ was arbitrary, this completes the proof of~\eqref{eq:LDEabove}, and hence of the theorem.
\end{proof}


\section{Convergence towards an asymptotic shape}\label{sec:mainproofs}

Recall that for any non-negative random variable $X$ and $p>0$, finiteness of $\E[X^p]$ is equivalent to convergence of the series $\sum_{n=1}^\infty n^{p-1}\,\Pr(X>n)$. More precisely, moments of $X$ can be expressed as
$$
\E[X^p]\,=\,p\int_0^\infty x^{p-1}\,\Pr(X>x)\,dx.
$$
With this in mind, we will easily be able to derive Theorem~\ref{thm:HREcone} and~\ref{thm:Lp-ST} from Theorem~\ref{thm:LDEcone}.

Although Theorem~\ref{thm:HREcone} and~\ref{thm:Lp-ST} were stated in terms of Euclidean distance, we will in the proofs below work with $\ell^1$-distances. Due to the equivalence between $\ell^p$-norms, this implies no restriction to the validity of arguments. However, it does facilitate summation over lattice points.

\begin{proof}[{\bf Proof of Theorem~\ref{thm:HREcone}}]
We first note that $\E[Y^p]<\infty$ is necessary. Let $\Hc=\Hc(u,c)$ and recall that a lower bound on $T_\Hc(0,z)-\muZ(z)$ is given by $Y-\mu(z)\ge Y-\muZ(\ebf_1)\|z\|$. Note further that the number of points in $\iVH$ at $\ell^1$-distance $n$ from the origin is of order $n^{d-1}$, meaning that it is bounded below by $\delta n^{d-1}$ and above by $\frac{1}{\delta}n^{d-1}$ for some $\delta>0$. Consequently,
$$
\sum_{z\in\iVH}\|z\|^{p-d}\,\Pr\big(T_\Hc(0,z)-\muZ(z)>\eps\|z\|\big)\;\ge\;\delta\sum_{n=1}^\infty n^{p-1}\,\Pr\big(Y>(\muZ(\ebf_1)+\eps)n\big),
$$
which is finite only if $\E[Y^p]<\infty$.

We next consider the sufficiency of $\E[Y^p]<\infty$. Let $q=p+1$ and $M=M(\eps,d,p,u)$ be given as in Theorem~\ref{thm:LDEcone}, by which we obtain that
\bea
\begin{aligned}
\sum_{z\in\iVH}\|z\|^{p-d}\,\Pr\big(|T_\Hc(0,z)-\muZ(z)|>\eps\|z\|\big)\;&\le\;
M\sum_{z\in\iVH}\Big(\|z\|^{p-d}\,\Pr(Y>\|z\|/M)+\|z\|^{-(d+1)}\Big)\\
&\le\;\frac{M}{\delta}\sum_{n=1}^\infty n^{p-1}\,\Pr(Y>n/M)+\frac{M}{\delta}\sum_{n=1}^\infty n^{-2}.
\end{aligned}
\eea
Here, in the latter step, we have again used the fact that the number of sites in $\iVH$ at distance $n$ from the origin is of order $n^{d-1}$. Sufficiency of $\E[Y^p]<\infty$ is now immediate.
\end{proof}

\begin{proof}[{\bf Proof of Theorem~\ref{thm:Lp-ST}}]
Let $\Hc=\Hc(u,c)$. Since $T_\Hc(0,z)$ is bounded below by $Y$ for $z\neq0$, the expectation of $\big|T_\Hc(0,z)-\muZ(z)\big|^p$ is infinite unless $\E[Y^p]<\infty$. It remains to show that $\E[Y^p]<\infty$ is also sufficient for $L^p$-convergence. For every $\eps>0$
\bea
\begin{aligned}
\E\left|\frac{T_\Hc(0,z)-\muZ(z)}{\|z\|}\right|^p\;&=\;p\int_0^\infty x^{p-1}\,\Pr\big(|T_\Hc(0,z)-\muZ(z)|>x\|z\|\big)\,dx\\
&\le\;p\int_0^\eps x^{p-1}\,dx+p\int_\eps^\infty x^{p-1}\,\Pr\big(|T_\Hc(0,z)-\muZ(z)|>x\|z\|\big)\,dx\\
&=\;\eps^p+p(\eps/\|z\|)^p\int_{\|z\|}^\infty y^{p-1}\,\Pr\big(|T_\Hc(0,z)-\muZ(z)|>\eps y\big)\,dy,
\end{aligned}
\eea
where we in the final step have made the substitution $x\|z\|=\eps y$. Let $q=p+1$ and $M=M(\eps,d,p,u)$ be as in Theorem~\ref{thm:LDEcone}. A further upper bound on the above expression is then
$$
\eps^p+pM(\eps/\|z\|)^p\int_{\|z\|}^\infty\Big(y^{p-1}\,\Pr(Y>y/M)+y^{-2}\Big)\,dy,
$$
which is finite when $\E[Y^p]<\infty$. Thus, sending $\|z\|$ to infinity, we obtain that
$$
\limsup_{z\in\iVH:\,\|z\|\to\infty}\E\left|\frac{T_\Hc(0,z)-\muZ(z)}{\|z\|}\right|^p\,\le\,\eps^p.
$$
Since $\eps>0$ was arbitrary, this ends the proof.
\end{proof}

\begin{remark}
Proposition~\ref{prop:regime} may be derived in a similar manner as Theorem~\ref{thm:HREcone}. In a first step one verifies that an analogous statement to that of Theorem~\ref{thm:LDEcone} also holds for the graph described in Proposition~\ref{prop:regime}. (The set of interior points may in this case be taken to include all sites $(z_1,z_2)$ with $z_1\ge e^{3/a}$ whose four neighbours are all contained in the graph.) Next, one can mimic the proof of sufficiency in Theorem~\ref{thm:HREcone} to obtain the almost sure convergence.
\end{remark}

\section{Dynamical stability and the shape theorem}\label{sec:dfpp}

In this section we consider dynamical first-passage percolation, with the aim at proving that the almost sure convergence towards the asymptotic shape is stable with respect to the introduced dynamics. Travel times in the dynamic environment can be studied via a comparison to non-dynamic environments, as we shall do. Recall that $\{\tau_e(s)\}_{e\in\Ec_{\Z^d}}$ denotes the i.i.d.\ family of weights evolving over time associated with the edges of the $\Z^d$ lattice. Let $N_e=N_e(\delta)$ denote the number or updates of $\tau_e(s)$ during the interval $[0,\delta]$. Clearly $N_e\sim\textup{Poisson}(\delta)$ and $\E[N_e]=\delta$. Define
$$
\bar{\tau}_e:=\tau_e(0)\cdot1_{\{N_e(\delta)=0\}},\quad\text{and}\quad\hat{\tau}_e:=\sup_{s\in[0,\delta]}\tau_e(s).
$$
Given a unit vector $u\in\R^d$ and $c>0$, let $\bar{T}_{\Hc(u,c)}(y,z)$ denote the travel time between $y$ and $z$ on $\Hc(u,c)$ with respect to $\{\bar{\tau}_e\}_{e\in\Ec_{\Z^d}}$, and likewise for $\hat{T}_{\Hc(u,c)}(y,z)$ with respect to $\{\hat{\tau}_e\}_{e\in\Ec_{\Z^d}}$. Clearly, for every $\delta\ge0$ and $y,z\in\Z^d$,
\be\label{eq:dynamicdomination}
\bar{T}_{\Hc(u,c)}(y,z)\;\le\; T_{\Hc(u,c)}^{(s)}(y,z)\;\le\;\hat{T}_{\Hc(u,c)}(y,z),\quad\text{for all }s\in[0,\delta].
\ee
The next lemma shows that they have all finite moments simultaneously.

\begin{lma}\label{lma:dynamicmoment}
For every $\delta>0$, $p>0$ and $q\geq1$
$$
\E\Big[\min\Big(\sup_{s\in[0,\delta]}\tau_1(s),\ldots,\sup_{s\in[0,\delta]}\tau_q(s)\Big)^p\Big]\;\leq\;(1+\delta)^q\E\Big[\min\big(\tau_1(0),\ldots,\tau_q(0)\big)^p\Big],
$$
where $\tau_1(s),\tau_2(s),\ldots,\tau_q(s)$ are i.i.d.\ and distributed as $\tau_e(s)$.
\end{lma}

\begin{proof}
Since $\E[X^p]=p\int_0^\infty x^{p-1}\Pr(X>x)\,dx$ and $\Pr\big(\min(X_1,\ldots,X_m)>x\big)=\Pr(X_1>x)^m$ for i.i.d.\ non-negative random variables, the result follows from the following observation.
\bea
\begin{split}
\Pr\Big(\sup_{s\in[0,\delta]}\tau_e(s)>x\Big)\;&=\;\sum_{k=0}^\infty \Pr\Big(\max_{j=1,\ldots,k+1}\tau_e^{(j)}>x\,\Big|N_e=k\Big)\Pr(N_e=k)\\
&\leq\;\sum_{k=0}^\infty(k+1)\Pr(\tau_e>x)\Pr(N_e=k)\;=\;\Pr(\tau_e>x)\E[1+N_e].\qedhere
\end{split}
\eea
\end{proof}

We will derive Theorem~\ref{thm:DST} from Theorem~\ref{thm:HREcone}, via~\eqref{eq:dynamicdomination} and the continuity of the time constant with respect to the weight distribution.

\begin{proof}[{\bf Proof of Theorem~\ref{thm:DST}}]
The leftwards implication is an immediate consequence of Theorem~\ref{thm:HREcone}, so it suffices to prove the rightwards going one. Fix $\eps>0$ and $p>0$, and let $\Hc=\Hc(u,c)$. Let also $\hat{Y}$ denote the minimum of $2d$ independent variables distributed as $\hat{\tau}_e$, and let $\bar{\mu}$ and $\hat{\mu}$ denote the time constants for the lattice $\Z^d$ associated with the distribution functions $\bar F_\delta(x)=\Pr(\bar\tau_e\le x)$ and $\hat F_\delta(x)=\Pr(\hat\tau_e\le x)$.

As $\delta$ tends to zero, both $\bar F_\delta$ and $\hat F_\delta$ converge weakly to the distribution of $\tau_e$. Hence, by~\eqref{eq:mucontinuity} we can choose $\delta=\delta(\eps)>0$ small enough for
$$
\big|\hat{\mu}(z/|z|)-\bar{\mu}(z/|z|)\big|<\eps/4\quad\text{for every }z\in\Z^d.
$$
Together with~\eqref{eq:dynamicdomination} we thus obtain for $x\ge|z|$ and for small enough $\delta>0$ that
\bea
\begin{aligned}
\Pr\Big(\sup_{s\in[0,\delta]}|T^{(s)}_\Hc(0,z)-\muZ(z)|>\eps x\Big)\;&\le\;\Pr\Big(\hat{T}_\Hc(0,z)-\muZ(z)>\frac{\eps x}{2}\Big)+\Pr\Big(\bar{T}_\Hc(0,z)-\muZ(z)<-\frac{\eps x}{2}\Big)\\
&\le\;\Pr\Big(\hat{T}_\Hc(0,z)-\hat\mu(z)>\frac{\eps x}{4}\Big)+\Pr\Big(\bar{T}_\Hc(0,z)-\bar\mu(z)<-\frac{\eps x}{4}\Big).
\end{aligned}
\eea
According to Lemma \ref{lma:dynamicmoment}, $\E[Y^p]<\infty$ implies that also $\E[\hat{Y}^p]<\infty$, and Theorem~\ref{thm:HREcone} shows that
$$
\E[\hat{Y}^p]<\infty\quad\Rightarrow\quad\sum_{z\in\iVH}|z|^{p-d}\,\Pr\Big(\sup_{s\in[0,\delta]}|T^{(s)}_\Hc(0,z)-\muZ(z)|>\eps|z|\Big)<\infty.
$$
The rightwards implication of the theorem is thus obtained by covering $[0,1]$ with finitely many intervals of length $\delta$.
\end{proof}

\section{Extending the shape theorem to the boundary}\label{sec:discussion}

We end this paper with a discussion under which conditions the convergence towards an asymptotic shape can be extended to hold for all points of the graph $\Hc=\Hc(u,c)$ simultaneously. Due to the equivalence between Euclidean and $\ell^1$-norms, we choose in this section to work directly with the latter. We begin with a weak law obtained under a minimal moment condition on the weight distribution.

\begin{prop}\label{prop:weakST}
Assume that $\E[Y^p]<\infty$ for some $p>0$. Then, for any $\eps>0$, unit vector $u\in\R^d$, and $c>0$,
$$
\limsup_{z\in\VH:\,\|z\|\to\infty}\Pr\big(|T_{\Hc(u,c)}(0,z)-\muZ(z)|>\eps\|z\|\big)=0.
$$
\end{prop}

\begin{proof}
First observe that it is boundary points that we need to worry about, since the superior limit taken over interior points follows via Theorem~\ref{thm:LDEcone}. Next, note that each boundary point $z$ is within Euclidean distance $5\sqrt{d}$ from some interior point $v_z$, by construction. Since $z$ and $v_z$ are within Euclidean distance $5\sqrt{d}$, they are also within $\ell^1$-distance $5d\sqrt{d}$, and there is thus a path $\Gamma$ of no greater length connecting the two. Since $T_\Hc(z,v_z)\le T(\Gamma)$, and $T(\Gamma)>x$ implies that $\tau_e>x/|\Gamma|$ for some $e\in\Gamma$, it follows that
$$
\Pr\big(T_\Hc(z,v_z)>x\big)\,\le\,5d\sqrt{d}\,\Pr\big(\tau_e>x/(5d\sqrt{d})\big).
$$

We now argue that the superior limit taken over boundary points also equals zero. Choose $p>0$ be such that $\E[Y^p]<\infty$, and fix $\eps>0$ and $q=1$. Let $M=M(\eps,d,p,u)$ be given as in Theorem~\ref{thm:LDEcone}. Then, for every $z\in\bVH$
$$
|T_\Hc(0,z)-\muZ(z)|\,\le\,|T_\Hc(0,v_z)-\muZ(v_z)|+T_\Hc(v_z,z)+|\muZ(v_z)-\muZ(z)|.
$$
By Lipschitz continuity of $\muZ$ we have $|\muZ(v_z)-\muZ(z)|\le5d\sqrt{d}\,\muZ(\ebf_1)$. So, combining the above estimate and Theorem~\ref{thm:LDEcone} we have for $z\in\bVH$ sufficiently far from the origin
\bea
\begin{aligned}
\Pr\big(|T_\Hc(0,z)-\muZ(z)|>3\eps\|z\|\big)\;&\le\;\Pr\big(|T_\Hc(0,v_z)-\muZ(v_z)|>\eps\|z\|\big)+\Pr\big(T_\Hc(z,v_z)>\eps\|z\|\big)\\
&\le\;M\,\Pr(Y>\|z\|/M)+\frac{M}{\|z\|}+5d\sqrt{d}\,\Pr\big(\tau_e>\eps\|z\|/(5d\sqrt{d})\big),
\end{aligned}
\eea
which clearly vanishes as $\|z\|\to\infty$.
\end{proof}

\subsection{A necessary and sufficient condition for convergence in half-space}

The precise condition for stronger modes of convergence turns out to vary quite a bit with the choice of $u$ and $c$. The reason being the varying number of neighbours among boundary points. Let us first consider the case when the underlying graph is a half-space. For $q\ge1$, let $Y_q$ denote the minimum of $q$ independent random variables distributed as $\tau_e$.

\begin{prop}
Let $\Hc=\Hc(\ebf_1,1)$. For every $\eps>0$ and $p\in(0,2d)$
$$
\E[Y^p]<\infty\quad\Leftrightarrow\quad\sum_{z\in\VH}\|z\|^{p-d}\,\Pr\big(|T_\Hc(0,z)-\muZ(z)|>\eps\|z\|\big)<\infty.
$$
For $p\ge2d$, finiteness of $\E[Y^p]$ and $\E[Y_{2d-1}^{p-1}]$ is both necessary and sufficient.
\end{prop}

\begin{proof}
It is immediate from Theorem~\ref{thm:HREcone} that $\E[Y^p]<\infty$ is both necessary and sufficient for the sum over interior points to be finite. It therefore suffices to consider points on the boundary.

Let us first argue for the necessity of $\E[Y_{2d-1}^{p-1}]<\infty$ when $p\ge2d$. Let $Y_\Hc(z)$ denote the minimum among the edge-weights associated to the edges adjacent to $z$ (and present in $\Hc$). A lower bound on $|T_\Hc(0,z)-\muZ(z)|$ is thus given by $Y_\Hc(z)-\muZ(\ebf_1)\|z\|$. Consequently
$$
\sum_{z\in\bVH}\|z\|^{p-d}\,\Pr\big(|T_\Hc(0,z)-\muZ(z)|>\eps\|z\|\big)\ge\sum_{z\in\bVH}\|z\|^{p-d}\,\Pr\big(Y_\Hc(z)>(\muZ(\ebf_1)+\eps)\|z\|\big).
$$
Since the number of boundary points at $\ell^1$-distance $n$ from the origin with precisely $2d-1$ neighbours is of order $n^{d-2}$, it follows that $\E[Y_{2d-1}^{p-1}]<\infty$ is necessary for summability.

We continue proving sufficiency of $\E[Y_{2d-1}^{p-1}]<\infty$, together with $\E[Y^p]<\infty$. For each boundary point $z\in\bVH$, let $v_z$ denote its projection to the hyper-plane $\{(z_1,\ldots,z_d)\in\Z^d:z_1=0\}$. Note that there are $2d-1$ disjoint paths of length at most $\|v_z-z\|+2$ between $z$ and $v_z$. As in the proof of Proposition~\ref{prop:weakST} we obtain for $z\in\bVH$ sufficiently far from the origin that
$$
\Pr\big(|T_\Hc(0,z)-\muZ(z)|>3\eps\|z\|\big)\;\le\;\Pr\big(|T_\Hc(0,v_z)-\muZ(v_z)|>\eps\|z\|\big)+\Pr\big(T_\Hc(z,v_z)>\eps\|z\|\big).
$$
Since $\|v_z-z\|$ is at most $4\sqrt{d}$, the $2d-1$ paths between $z$ and $v_z$ have length at most $6\sqrt{d}$. Just as in the proof of Proposition~\ref{prop:Ycond} we find that $(6\sqrt{d})^{2d-1}\,\Pr\big(Y_{2d-1}>x/(6\sqrt{d})\big)$ is an upper bound for $\Pr\big(T_\Hc(z,v_z)>x\big)$. Since the number of boundary points at $\ell^1$-distance $n$ from the origin grows at order $n^{d-2}$, we conclude that
\bea
\begin{aligned}
\sum_{z\in\bVH}\|z\|^{p-d}\,\Pr\big(|T_\Hc(0,z)-\muZ(z)|>3\eps\|z\|\big)\;&\le\;\sum_{z\in\bVH}\|z\|^{p-d}\,\Pr\big(|T_\Hc(0,v_z)-\muZ(v_z)|>\eps\|z\|\big)\\
&\quad\;+(6\sqrt{d})^{2d-1}\sum_{n=1}^\infty n^{p-2}\,\Pr\big(Y_{2d-1}>\eps n/(6\sqrt{d})\big).
\end{aligned}
\eea
Assume that $\E[Y^p]<\infty$. An application of Theorem~\ref{thm:HREcone} (or Theorem~\ref{thm:LDEcone}) shows that the first sum in the right-hand side is finite. Note that the second sum is trivially finite for $p<1$; For $p\ge1$ a sufficient condition is that $\E[Y_{2d-1}^{p-1}]<\infty$.

It now remains to show that $\E[Y^p]<\infty$ implies $\E[Y_{2d-1}^{p-1}]<\infty$ for $p<2d$. Via Markov's inequality we see that
$$
\Pr(Y_{2d-1}>x)\;=\;\Pr(\tau_e>x)^{2d-1}\;=\;\Pr(Y>x)^{1-1/2d}\;\le\;\big(x^{-p}\,\E[Y^p]\big)^{1-1/2d}.
$$
As a consequence,
$$
\E[Y_{2d-1}^{p-1}]\;=\;(p-1)\int_0^\infty x^{p-2}\,\Pr(Y_{2d-1}>x)\,dx\;\le\;(p-1)\Big(1+\E[Y^p]^{1-1/2d}\int_1^\infty x^{p/2d-2}\,dx\Big),
$$
which is finite if $\E[Y^p]<\infty$ and $p<2d$.
\end{proof}

\subsection{A general sufficient condition}

As above, let $Y_\Hc(z)$ denote the minimum weight among the edges adjacent to $z\in\VH$. A fundamental restriction on a relation between summability of tail probabilities over boundary points of $\Hc=\Hc(u,c)$ and moments of the weight distribution is given by the following inequality
\be\label{eq:Ylower}
Y_\Hc(z)-\mu(z)\,\le\,|T_\Hc(0,z)-\muZ(z)|.
\ee
Consequently, summability is directly related to the number of neighbours of a site. Let us first note that the number of boundary points with only one neighbour may be large. For $d=3$, let $u=(1,1,1)/\sqrt{3}$ and $b\ge4\sqrt{3}$ be an integer. Then $c>0$ may be chosen such that the graph $\Gc$ induced by $\bigcup_{a\ge0}B(au,ca+b)$ contains every point along the diagonal $\{(z_1,z_2,z_3)\in\Z^3:z_1=z_2\ge0,z_3=-b\}$, but no other point in the plane $\{z_3=-b\}$. Each point on the diagonal has only one neighbour in $\Gc$. Since the number of such points at distance $n$ from the origin is positive, it follows via~\eqref{eq:Ylower} that $\E[\tau_e^p]<\infty$ is necessary for $\sum_{z\in\Vc_\Gc}\|z\|^{p-1}\,\Pr\big(|T_\Gc(0,z)-\muZ(z)|>\eps\|z\|\big)$ to be finite. A more careful investigation of the number of neighbours to boundary points reveals a complementary sufficient condition: For every $\eps>0$, $c>0$, unit vector $u\in\R^d$ and $p\ge1$,
$$
\E[\tau_e^p]<\infty\quad\Rightarrow\quad\sum_{z\in\VH}\|z\|^{p-1}\,\Pr\big(|T_{\Hc(u,c)}(0,z)-\muZ(z)|>\eps\|z\|\big)<\infty.
$$
A proof of the latter statement requires a somewhat technical geometric argument, but is for completeness included in the appendix.

\begin{acknow}
The author would like to thank Olle H\"aggstr\"om for suggesting the dynamical version of first-passage percolation, as well as Robert Morris and Graham Smith for discussing some geometrical issues.
\end{acknow}

\appendix

\section{Proof of the general sufficient condition for summability}

We here prove the following general statement.

\begin{prop}\label{prop:generalboundary}
For every $\eps>0$, $c>0$, unit vector $u\in\R^d$ and $p\ge1$
$$
\E[\tau_e^p]<\infty\quad\Rightarrow\quad\sum_{z\in\VH}\|z\|^{p-1}\,\Pr\big(|T_{\Hc(u,c)}(0,z)-\muZ(z)|>\eps\|z\|\big)<\infty.
$$
\end{prop}

We will need another geometrical lemma. Let
$$
H_n:=\Big\{x\in\R^d:\|x\|_\infty:=\max_{i=1,2,\ldots,d}|x_i|=n\Big\}.
$$

\begin{lma}\label{lma:partition}
There is a partition $D_1,D_2,\ldots,D_{d-1}$ of the vertices in $\bVH$, such that for some $M<\infty$,
\bea
\big|D_q\cap H_n\big|\,\le\,Mn^{q-1}\quad\text{for }q=1,2,\ldots,d-1,
\eea
and if $z\in D_q\cap H_n$ for some $n$ and $q\ge1$, then there is a $v_z\in D_0\cap H_n$ such that $\|v_z-z\|\le M$, and there exist (at least) $q$ disjoint paths between $z$ and $v_z$ of length $\|v_z-z\|$. 
\end{lma}

\begin{proof}
We will prove the case when $u=\ebf_1$, and leave the remaining details to the reader. There are three cases: $c>1$, $c=1$, and $c<1$. In the first case there is nothing to prove. In the second case set $D_{d-1}=\bVH$. The final case is more tricky.

Thus, assume that $c<1$. Since $H_n$ is a $(d-1)$-dimensional subset of $\R^d$, then there is $M_1<\infty$ such that
$$
\big|\bVH\cap H_n\big|\,\le\, M_1n^{d-2}.
$$
For $q=1,2,\ldots,d-1$ we define
$$
D_q:=\big\{z\in\bVH:\,z_i\neq0\text{ for $q$ indices }i\ge2\big\}.
$$
It is clear that fixation of $z_i=0$ for some $i\ge2$ in $\bVH$ reduces the degree of freedom (in the choice of $z$) by one, in the sense that going from $D_{q+1}$ to $D_q$ we loose one dimension. Hence, there is $M_2<\infty$
$$
|D_q\cap H_n|\,\le\, M_2n^{q-1},\quad\text{for each }q=1,2,\ldots,d-1.
$$

Move on to the second part of the statement. Take $z\in D_q\cap H_n$. Due to lattice symmetry, we may assume that $z_i\ge0$ for all $i=1,2,\ldots,d$. We will choose $v_z$ suitably in the rectangle
$$
\mathcal{R}_z=\big\{v\in\Vc_\Gc\cap H_n:\;z_1\le v_1\le n,\text{ and } 0\le v_i\le z_i\text{ for }i\ge2\big\}.
$$
From the definition of interior points we can find $M_3<\infty$ such that, for $v\in\mathcal{R}_z$, if $z_1+M_3\le v_1\le n$, or if $v_1=n$ and $|v-n\ebf_1|\le|z-n\ebf_1|-M_3$, then $v\in\iVH$. Choose $v_z$ accordingly. Since $v_z$ and $z$ differ in $q$ coordinates, it is easy to find $q$ disjoint paths from $v_z$ to $z$ of length $\|v_z-z\|$.
\end{proof}

Recall that $Y_q$ denotes the minimum of $q$ independent random variables distributed as $\tau_e$.

\begin{lma}\label{lma:appendix}
For every $p>0$ and integer $q\ge1$
$$
\E[\tau_e^p]<\infty\quad\Rightarrow\quad\E[Y_q^{pq}]<\infty.
$$
\end{lma}

\begin{proof}
The statement follows from a simple application of Markov's inequality: $\Pr(\tau_e>x)\le x^{-p}\,\E[\tau_e^p]$, to be used as follows.
\bea
\begin{aligned}
\E[\min\{\tau_1,\ldots,\tau_q\}^{pq}]\;&=\;pq\int_0^\infty x^{pq-1}\,\Pr(\tau_e>x)^q\,dx\\
&\le\;pq\int_0^\infty x^{pq-1}\big(x^{-p}\,\E[\tau_e^p]\big)^{q-1}\,\Pr(\tau_e>x)\,dx\\
&=\;pq\E[\tau_e^p]^{q-1}\int_0^\infty x^{p-1}\,\Pr(\tau_e>x)\,dx\;=\;q\E[\tau_e^p]^q,
\end{aligned}
\eea
as required.
\end{proof}

We may now prove the remaining statement.

\begin{proof}[{\bf Proof of Proposition~\ref{prop:generalboundary}}]
Assume that $\E[\tau_e^p]<\infty$ for some $p\ge1$. Consequently $\E[Y^{2dp}]$ is finite according to Lemma~\ref{lma:appendix}, and summability over interior points follows from Theorem~\ref{thm:HREcone}. It thus suffices to consider the contribution of boundary points.

Let $v_z$, $M<\infty$ and $D_1,D_2,\ldots,D_{d-1}$ be given as in Lemma~\ref{lma:partition}. By subadditivity we have
$$
|T_\Hc(0,z)-\muZ(z)|\,\le\,|T_\Hc(0,v_z)-\muZ(v_z)|+T_\Hc(v_z,z)+|\muZ(z)-\muZ(v_z)|.
$$
By definition $\|v_z-z\|\le M$, so $|\muZ(z)-\muZ(v_z)|\le\muZ(\ebf_1)M$ due to Lipschitz continuity of $\muZ$. Hence, for every $\eps>0$ we have for all $z\in\bVH$ sufficiently far (depending on $\eps$) from the origin that
$$
\Pr\big(|T_\Hc(0,z)-\muZ(z)|>3\eps\|z\|\big)\,\le\,\Pr\big(|T_\Hc(0,v_z)-\muZ(v_z)|>\eps\|z\|\big)+\Pr\big(T_\Hc(v_z,z)>\eps\|z\|\big).
$$
The latter term in the right-hand side is for $z\in D_q$ at most $\|v_z-z\|^q\,\Pr(Y_q>\eps\|z\|/M)$. Since the number of points in $D_q$ at $\ell^1$-distance $n$ from the origin is (at most) of order $n^{q-1}$, we arrive at the following upper bound
\bea
\begin{aligned}
\sum_{z\in\bVH}&\|z\|^{p-1}\,\Pr\big(|T_\Hc(0,z)-\muZ(z)|>3\eps\|z\|\big)\\
&\le\;\sum_{z\in\bVH}\|z\|^{p-1}\,\Pr\big(|T_\Hc(0,v_z)-\muZ(v_z)|>\eps\|z\|\big)+\sum_{q=1}^{d-1}\sum_{z\in D_q}\|z\|^{p-1}\,\Pr\big(T_\Hc(v_z,z)>\eps\|z\|\big)\\
&\le\;\sum_{z\in\bVH}\|z\|^{p-1}\,\Pr\big(|T_\Hc(0,v_z)-\muZ(v_z)|>\eps\|z\|\big)+\sum_{q=1}^{d-1}M^q\sum_{n=1}^\infty n^{p+q-2}\,\Pr(Y_q>\eps n/M).
\end{aligned}
\eea
The former sum on the lower line above is via Theorem~\ref{thm:HREcone} finite if $\E[Y^{p+d-1}]$ is finite. The latter (double-) sum is finite if $\E[Y_q^{p+q-1}]$ is finite for all $q=1,2,\ldots,d$. Since $\E[\tau_e^p]$ was assumed finite, also the remaining conditions are satisfied as a consequence of Lemma~\ref{lma:appendix}.
\end{proof}

\bibliographystyle{alpha}
\bibliography{bibpercolation}

\end{document}